\documentclass[12pt]{amsart}
\usepackage[margin=24mm]{geometry}
\usepackage[skins]{tcolorbox}
\usepackage{blindtext}
\usepackage{graphicx}
\usepackage{amssymb}
\usepackage{amsfonts}
\usepackage{amsthm}
\usepackage{amsmath}
\usepackage{eucal}
\usepackage{pdfsync}
\usepackage{hyperref}
\usepackage{enumerate}
\usepackage{mathrsfs}
\usepackage{dsfont}
\usepackage{textcomp}
\usepackage{tipa}
\usepackage{mathtools}
\usepackage[english]{babel}
\usepackage[autostyle]{csquotes}
\usepackage{xcolor}
\usepackage{enumerate}
\usepackage{upgreek}
\usepackage{tikz}
\usetikzlibrary{matrix}

\newtheorem{lem}{Lemma}
\newtheorem{prop}{Proposition}[section]
\newtheorem{thm}{Theorem}[section]
\newtheorem*{thmnol}{Theorem}
\newtheorem{cor}{Corollary}

\newtheorem{rem}{Remark}

\tcolorboxenvironment{thm}{blanker,before skip=10pt,after skip=10pt}
\tcolorboxenvironment{lem}{blanker,before skip=10pt,after skip=10pt}
\tcolorboxenvironment{prop}{blanker,before skip=10pt,after skip=10pt}
\tcolorboxenvironment{cor}{blanker,before skip=10pt,after skip=10pt}

\theoremstyle{remark}

\title[Growth of torsion]{Growth of Odd Torsion Over Imaginary Quadratic Fields of Class Number 1 }

\author{Irmak Bal\c{c}{\i}k} 
\address{Department of Mathematics, University of Southern California}
\email{balcik@usc.edu}

\begin{document}

\maketitle

\begin{abstract}
	Let $K$ be a non-cylotomic imaginary quadratic field of class number 1 and $E/K$ is an elliptic curve with $E(K)[2]\simeq \mathbb{Z}_1.$ 
    We determine the odd-order torsion groups that can arise as $E(L)_{\text{tor}}$ where $L$ is a quadratic extension of $K.$
\end{abstract}

\date{\today}

\section{Introduction}
The possible torsion groups of an elliptic curve over the rational numbers are known by Mazur's theorem \cite{Maz77}. These groups are
\begin{equation} \label{eq:mazurgroups}
\begin{array}{ll}
\mathbb{Z}_n, & 1\leq n \leq 12,\  n\neq 11 \\
\mathbb{Z}_2\oplus \mathbb{Z}_{2n}, & 1\leq n\leq 4
\end{array}
\end{equation}

The work of Mazur was generalized by Kenku, Momose and Kamienny to quadratic fields. 

\begin{thmnol}[Kamienny \cite{Kam92}, Kenku and Momose \cite{KM88}] \label{thm2}
Let K be a quadratic number field and $E/K$ an elliptic curve. Then $E(K)_{\text{tor}}$ is isomorphic to one of the followings:
\begin{equation} \label{eq:kamiennygroups}
\begin{array}{ll}
 \mathbb{Z}_n, &  1\leq n\leq 18,\ n\neq 17  \\  \mathbb{Z}_2 \oplus \mathbb{Z}_{2n}, & 1\leq n\leq 6  \\ 
\mathbb{Z}_3 \oplus \mathbb{Z}_{3n}, & 1\leq n\leq 2 \\ 
\mathbb{Z}_4 \oplus \mathbb{Z}_4 & . 
\end{array}
\end{equation}
\end{thmnol}

Over cubic fields, the exact analog of Mazur's theorem has been established last year \cite{DEVMB21}.  Although it is not known which groups can appear as torsion groups of elliptic curves over quartic number fields, Jeon, Kim and Park \cite{JKP06} determined which groups can appear for infinitely many non-isomorphic curves.
It was also shown by Derickx, Kamienny, Stein and Stoll \cite{DKSS17} that $17$ is the largest prime dividing $|E(K)_{\text{tor}}|$ over any quartic field $K.$
This work is a first step towards classifying the growth of torsion in quadratic extensions of base fields larger than $\mathbb{Q}$. Specifically we are interested in the following question: if $K$ is a quadratic field and $E/K$ is an elliptic curve with $|E(K)_{\text{tor}}|$  odd, how does the torsion subgroup of $E(K)$ grow upon quadratic extension? We answer this question completely for the set of number fields
$ \mathcal{S}=\{ \mathbb{Q}(\sqrt{-2}), \mathbb{Q}(\sqrt{-7}), \mathbb{Q}(\sqrt{-11}),\mathbb{Q}(\sqrt{-19}),\mathbb{Q}(\sqrt{-43}),\mathbb{Q}(\sqrt{-67}), \mathbb{Q}(\sqrt{-163})\}$
consisting of all non-cyclotomic imaginary quadratic fields of class number one. 


\begin{thm}\label{mainthm}
Let $K \in \mathcal{S}$  and $E/K$ an elliptic curve with $E(K)[2]\simeq \mathbb{Z}_1,$ Then, for any extension $L$ with $[L:K]=2,$ $E(L)_{\text{tor}}$ is isomorphic to one of the following groups:
	$$ \mathbb{Z}_n, \  n \in \{ 1,3,5,7,9,11,15\}, \  \mathbb{Z}_3 \oplus \mathbb{Z}_3, \ \mathbb{Z}_3 \oplus \mathbb{Z}_9.$$
Furthermore, $\mathbb{Z}_3 \oplus \mathbb{Z}_9$ is only realized in a quadratic extension of $K=\mathbb{Q}(\sqrt{-2})$ or $\mathbb{Q}(\sqrt{-11}).$
	
\end{thm}

\begin{table}[h!]
	\centering
	\caption{Growth of Odd Torsion of Ellliptic Curves Defined over $K \in \mathcal{S}$ }
	\begin{tabular}{ | p{2cm}|p{13.5cm}| }
		\hline
		$E(K)_{\text{tor}}$ & $E(L)_{\text{tor}}$ in some quadratic extension $L$ of $K$ (depending on $E$ and $K$) \\
		\hline
		$\mathbb{Z}_1$ & $\mathbb{Z}_1, \mathbb{Z}_3, \mathbb{Z}_5, \mathbb{Z}_7, \mathbb{Z}_9, \mathbb{Z}_{11}, \mathbb{Z}_{15}$  \\
		\hline
		$\mathbb{Z}_3$ & $\mathbb{Z}_3, \mathbb{Z}_{15}, \mathbb{Z}_3 \oplus \mathbb{Z}_3, \mathbb{Z}_3 \oplus \mathbb{Z}_9$  \\
		\hline
		$\mathbb{Z}_5 $ & $\mathbb{Z}_5, \mathbb{Z}_{15}$  \\
		\hline
		$\mathbb{Z}_7$ & $\mathbb{Z}_7$  \\
		\hline
		$\mathbb{Z}_9$ & $\mathbb{Z}_9, \mathbb{Z}_3 \oplus \mathbb{Z}_9$  \\
		\hline
		$\mathbb{Z}_{11} $ & $\mathbb{Z}_{11}$ \\
		\hline
		$\mathbb{Z}_{15}$ & $\mathbb{Z}_{15}$ \\
		\hline
	\end{tabular}
	\label{table:1}
\end{table}

Theorem \ref{mainthm} completes earlier work of Newman \cite{New16} who determined the growth of odd-torsion over the cyclotomic imaginary quadratic fields of class number one, that is $K = \mathbb{Q}(i)$ and $K = \mathbb{Q}(\sqrt{-3})$. His work is based on Najman's paper \cite{Naj11} in which the set of torsion groups arising over these particular fields $\mathbb{Q}(i)$ and $\mathbb{Q}(\sqrt{-3})$ are determined. The primary groups included in our list but do not arise in Newman's list are $\mathbb{Z}_{11}$ and $\mathbb{Z}_3 \oplus \mathbb{Z}_9.$ 

We choose to work over the imaginary quadratic fields of class number one for two main reasons (i) these are the only number fields besides $\mathbb{Q}$ for which a classification of torsion groups is known (see Proposition \ref{mainlist} below) (ii) the ring of integers of a class number one field is a principal ideal domain and its unit group is finite. These properties are rather useful for arithmetic computations. 

We now describe our strategy in more detail. Given $E$ an elliptic curve over a number field $K$ if the odd-torsion of $E(K)_{\text{tor}}$ grows over a quadratic extension $L$ of $K$ then there exists a $K$-rational point $P$ of odd-order arising on a quadratic twist $E^d$ of $E$. To start, we can assume $P$ has a prime order greater than 2. There are two restrictions on the existence of such a point $P.$ First, the order of $P$ must be one of the primes occuring as torsion over $K.$ Second, letting $H=\langle P \rangle $ we know that the subgroup $J=H\oplus E(K)_{\text{tor}}$ of $E(L)$ forms a cyclic $K$-rational $N$-isogeny where $N=|J|.$ (see Prop \ref{newman} below) 

Understanding of $K$-rational cyclic $N$-isogenies is equivalent to finding $K$-rational points on the modular curves $X_0(N)$. 
Indeed, the non-cuspidal $K$-rational points on $X_0(N)$ classify equivalence classes of pairs $(E,C)$ where $E/K$ is an elliptic curve and $C$ is a cyclic subgroup of order $N$ in $E(\overline{K})$ such that $C$ is invariant under the action of Gal$(\overline{K}/K).$ 
Two pairs $(E,C)$, $(E',C')$ are equivalent iff there exists an isomorphism $\phi : E \rightarrow E'$ such that $\phi(C)=C'.$ Thus, for $K$ a quadratic number field and $E$ an elliptic curve defined over $K$, the determination of the possible group structures for $E(L)_{\text{tor}}$ involves examining all quadratic points on certain modular curves. A major difficulty arises when the model of $X_0(N)$ is of high genus. This is the case for $N=77$ in which $X_0(77)$ has genus 7 and is the hardest case that we address. Results of Bars \cite{Bar99} and Harris-Silverman \cite{HS91} assert that if $X_0(N)$ is of genus greater of equal to 2, then it has finitely many quadratic points, except 28 values of $N$ which are \
$22$,$23$,$26$,$28$,$29$,$30$,$31$,$33$,$35$,$37$,$39$,$40$,$41$,$43$,$46$,$47$,$48$,$50$,$53$,$59$,$61$,$65$,$71$,$79$,$83$,$89$,$101$,and $131.$ In particular,
the set of quadratic points on $X_0(77)$ is finite. However, finding all the points is challenging and requires an application of relative symmetric Chabauty along with the Mordell-Weil sieve. This approach is inspired by earlier work of \cite{Box20} in which the highest genus handled is $5.$ We will use his method.

In the cases $N=45,55$ the curves $X_0(N)$ are non-hyperelliptic curves of genus $3,5$ respectively.
The quadratic points on these curves are determined by Siksek and Ozman \cite{OS18}. We compute whether each point lying on the curves gives rise to a cyclic $N$-isogeny in a quadratic extension of $K.$

In the cases $N = 33, 35$ the curves $X_{0}(N)$ are hyperelliptic of genus 3 with Mordell-Weil rank rk$(J_0(N))=0$ where $J_0(N)$ denotes the Jacobian of $X_0(N).$ As the set of $K$-rational points on $X_0(N)$ does not form a group structure, we instead study rational points on the symmetric square $X_0(N)^{(2)}$ of $X_0(N)$, which classifies effective divisors of degree 2 on $X_0(N).$ By the existence of degree 2 map from $X_0(N) \rightarrow \mathbb{P}^{1}$ over $\mathbb{Q},$ the set of quadratic points are classified into two subsets: the ones which arise as a pre-image of rational point on $\mathbb{P}^1$, called $non$-$exceptional$ and the ones which do not arise in this way, called $exceptional.$ 
Since $J_0(N)(\mathbb{Q})\ \textless\ \infty,$ the exceptional points are finite. Every non-zero $\mathbb{Q}$-point of $J_0(N)$ has a unique Mumford representation, see  Theorem 4.145 of \cite{CF06}. Magma lists points of $J_0(N)(\mathbb{Q})$ in Mumford representation. From these Mumford representations, we can extract (exceptional) points on $X_0(N).$ In order to find the non-exceptional points in the case $N=35,$ we additionally benefit from the  classification of points on quadratic twists of $X_0(N)$ determined in \cite{Ozman}.

In the cases $N = 21,27,$ $X_0(N)$ is an elliptic curve with  positive rank over some $K \in \mathcal{S}$ which makes the enumeration of all quadratic points computationally infeasible. We then shift our perspective and resolve these cases using the division polynomial method. 

\subsection*{Layout of the paper}
In Section 2, we present some known results about the growth of torsion in a quadratic extension. These results provide the context and motivation for the main results of this paper. Section 3 will contain the analysis of cyclic N-isogenies pointwise defined over a quadratic extension of $K \in \mathcal{S}$ which show the tools involved in the result of next section. In Section 4,  we classify $E^d(K)_{\text{tor}}$ where $E^d$ runs through all quadratic twists of a given elliptic curve $E/K$ with trivial $2$-torsion and in Section 5, we provide a proof to our main theorem. 

\subsection*{Notation}

Throughout this paper $K$ denotes a quadratic field. $L$ denotes a quadratic extension of $K$. Furthermore $\mathcal{S}$ stands for the set of non-cyclotomic imaginary quadratic fields of class number one. The rank and torsion subgroup computations are done by using the computer algebra software programs Magma and Sage.  

The code that verifies $X_0(77)$ computations can be found here:
$$ \href{https://github.com/irmak-balcik/X077}{https://github.com/irmak-balcik/X077} $$
      
\subsection*{Acknowledgement} The author is grateful 
to Sheldon Kamienny for his guidance and kind support as well as thankful to Samir Siksek for his insightful comments and sharing Magma codes for a part of the analysis of $X_0(77).$ 
This paper is also greatly benefited from the discussions with Noam Elkies, Michael Stoll (through mathoverflow) and many conversations with {\"O}zlem Ejder.

\section{Auxiliary Results}

As a prerequisite to proving Theorem \ref{mainthm} we need a classification of torsions over the base fields contained in 
$$ \mathcal{S}=\{ \mathbb{Q}(\sqrt{-2}), \mathbb{Q}(\sqrt{-7}), \mathbb{Q}(\sqrt{-11}),\mathbb{Q}(\sqrt{-19}),\mathbb{Q}(\sqrt{-43}),\mathbb{Q}(\sqrt{-67}), \mathbb{Q}(\sqrt{-163})\}.$$
Sarma and Saika \cite{SS18} determine the list of possible groups arising as torsion subgroups as $K$ varies in $\mathcal{S}$. The reader can notice that each list is strictly smaller than the list appearing in Kamienny and Kenku-Momose's list \eqref{eq:kamiennygroups}. The approach of \cite{SS18} depends on the method outlined in Kamienny and Najman's paper \cite{KN12}.

\begin{prop}[\cite{SS18}]\label{mainlist}
	Let $E$ be an elliptic curve over a given $K \in \mathcal{S}.$ Then $E(K)_{\text{tor}}$ is isomorphic to one of the Mazur's groups,
	\begin{enumerate}
		\item $\mathbb{Z}_{11}$ or $\mathbb{Z}_2 \oplus \mathbb{Z}_{10}$ when $K=\mathbb{Q}(\sqrt{-2}).$
		\item $\mathbb{Z}_{11}$, $\mathbb{Z}_{14}$ or $\mathbb{Z}_{15}$ when $K=\mathbb{Q}(\sqrt{-7}).$
		\item $\mathbb{Z}_{14}$, $\mathbb{Z}_{15}$ or $\mathbb{Z}_2\oplus \mathbb{Z}_{10}$ when $K=\mathbb{Q}(\sqrt{-11}).$
		\item $\mathbb{Z}_{11}$, $\mathbb{Z}_{2} \oplus \mathbb{Z}_{10}$ or $\mathbb{Z}_2 \oplus \mathbb{Z}_{12}$ when $K=\mathbb{Q}(\sqrt{-19}).$
		\item $\mathbb{Z}_{11}$, $\mathbb{Z}_{14}, \mathbb{Z}_{15}$ or $\mathbb{Z}_2 \oplus \mathbb{Z}_{12}$ when $K=\mathbb{Q}(\sqrt{-43}).$
		\item $\mathbb{Z}_{14}$, $\mathbb{Z}_{15}$ or $\mathbb{Z}_2 \oplus \mathbb{Z}_{12}$ when $K=\mathbb{Q}(\sqrt{-67}).$
		\item $\mathbb{Z}_{14}, \mathbb{Z}_{15}$ or $\mathbb{Z}_2 \oplus \mathbb{Z}_{12}$ when $K=\mathbb{Q}(\sqrt{-163}).$
	\end{enumerate} 
\end{prop}

Let $E$ be an elliptic curve given by $y^2 = f(x)$ with $f$ is a monic cubic polynomial over $K.$ There are a couple of observations we can make:

\begin{rem}\label{rem1}
	
(i) If $f$ is irreducible in $K$, then it remains irreducible over $L$. Otherwise, $f$ has
a root $\alpha$ in $L$ and the degree of $K(\alpha)$ over $K$ is divisible by $3$ but it is not possible
since $K(\alpha)$ is contained in $L$. Hence $E(K)[2]\simeq \mathbb{Z}_1$ then $E(L)[2]\simeq \mathbb{Z}_1.$

(ii) If $E(K)$ has a point of order $2$, then $E^d(K)$ has a point of order $2,$ too. This simply follows by if $f(x) = x^3 + ax^2 + bx + c$ for some $a,b,c \in K$ and for $d \in K$, $d\neq 0$ let $f^d(x) = x^3 +dax^2 +d^2bx+d^3c$. If $r \in K$ and $f(r)=0$, then $f^d(dr) = 0$ and this induces a bijection between the $K$-roots of $f$ and $f^d$. Hence $E(K)[2] \simeq E^d(K)[2].$ 
\end{rem}

\begin{prop}\label{prop1}
	Suppose $E$ is an elliptic curve over a number field $K$. Then for all but finitely many quadratic twists $E^d$ of $E$ with $d$ a non-square in $K,$ $E^d(K)$ has no odd torsion. 
\end{prop}

\begin{proof} 
	The proof is due Mazur and Rubin \cite[Lemma 5.5]{MR10}. 
\end{proof}

The next lemma plays a key role in the proof of the main theorem. It shows that the growth of odd-torsion can be understood with torsion arising over base field.

\begin{lem}(\cite{GJT14},Cor 4)\label{lem1}
	If n is an odd positive integer we have
	$$ E(K(\sqrt{d}))[n] \simeq E(K)[n] \oplus E^d(K)[n] $$
\end{lem}

An immediate corollary of the combination of these results follows.

\begin{cor}\label{cor6}
Let E be an elliptic curve defined over a quadratic number field $K$ with trivial 2-torsion. Then $E(K(\sqrt{d}))_{\text{tor}}$ is of odd order for all non-squares $d \in K$ and with only finitely many exceptions,
$$
E(K(\sqrt{d}))_{\text{tor}} = E(K)_{\text{tor}}
$$
\end{cor} 

\begin{proof}
	Suppose $E/K$ and $d$ satisfy the hypothesis. Since  $E(K)_{\text{tor}} \subseteq E(K(\sqrt{d}))_{\text{tor}},$ let us assume $E(K)_{\text{tor}} \subsetneq E(K(\sqrt{d}))_{\text{tor}}.$ Note that the claim  where $E(K(\sqrt{d}))[2] =\mathbb{Z}_1$ follows immediately from the observation in Remark \ref{rem1}. By Lemma \ref{lem1}, $E(K(\sqrt{d}))_{\text{tor}}$ is strictly larger than $E(K)_{\text{tor}}$ only if $E^d(K)[n]$ is non-trivial for some odd number $n.$ From Proposition \ref{prop1} , we know that there is only a finite number of quadratic twists of $E$ with a non-trivial odd torsion, which therefore yields finitely many quadratic extensions of $K$ in which $E(K)_{\text{tor}}$ grows as claimed.
\end{proof}

\begin{rem}
	Corollary \ref{cor6} can in fact be generalized to any given elliptic curve over a quadratic field. See Corollary 1 in \cite{Bal(2)21}.
	
\end{rem}

Corollary \ref{cor6} particularly shows that the list of odd-order torsion groups given in Sarma and Saika's classification must also appear as torsion group over a quadratic extension of base field.  Thus a large number of the torsion structures that we search for are known. 

The following result will be useful for determining the growth of odd torsion subgroup in a quadratic extension of a number field.

\begin{prop}(\cite[Proposition 3]{New16})\label{newman}
Let $K$ be a quadratic field and E an elliptic curve defined over $K$ with $|E(K)|=m$. Let $d \in K$ be a non-square and let $L=K(\sqrt{d}).$ If $H$ is a subgroup of $E^d(K)$ of odd order $n,$ then there is a Gal$(\overline{K}/K)$-invariant subgroup $J$ of $E(L)$ such that $J \simeq H \oplus E(K)_{\text{tor}}.$
\end{prop}

\section{Cyclic Isogenies}

Let $K$ be fixed in $\mathcal{S}.$ In this section, we classify $K$-rational cyclic $N$-isogenies of elliptic curves defined over $K$. The values of $N$ will vary according to the prime factors arising in the torsion subgroup over the base field. For a given elliptic curve $E$ defined over $K$, we recall that a $K$-rational $\bold{ cyclic\ N}$-$\bold{isogeny}$ of $E/K$ is a cyclic subgroup of $E(\overline{K})$ of order $N$ which is invariant under the action of Gal$(\overline{K}/K).$ 

For a given $N$-isogeny $C$, Magma describes a polynomial $f_C$ whose roots are precisely the $x$-coordinates of the points in $C$. Let $K(f_C)$ denote the splitting field of $f_C$ over $K.$ If $C$ is pointwise defined over a field $L$ then $f_C$ must split completely over $L$, i.e. $K(f_C)$ must be contained in $L.$ In particular, if $L$ is a quadratic extension of $K$, then $f_C$ must have irreducible factors of degree at most 2 over $K.$ 

With the study of quadratic points on various modular curves throughout this section, our goal is to obtain the following.

 \begin{thm}\label{thmc}
 	Let $K$ be a non-cyclotomic imaginary quadratic field of class number one and $E$ an elliptic curve defined over $K$. Then, $E$ has no cyclic N-isogeny defined over a quadratic extension of $K$ for $N=21,33,35,45,55,77$. 
 	In particular, a $27$-isogeny exists only in a quadratic extension of $K= \mathbb{Q}(\sqrt{-2})$ and $\mathbb{Q}(\sqrt{-11}).$ 
 	
 \end{thm}
 
\section*{  $N=15$ }

By Proposition \ref{mainlist}, $\mathbb{Z}_{15}$ is realized (infinitely many times) as torsion subgroup of elliptic curves defined over $K$ belonging to $\mathcal{E}= \{\mathbb{Q}(\sqrt{-7}),\mathbb{Q}(\sqrt{-11}),\mathbb{Q}(\sqrt{-43}), \mathbb{Q}(\sqrt{-67}), \mathbb{Q}(\sqrt{-163})\}.$ We study further the occurance of torsion subgroup of order 15 in a quadratic extension of each remaining field $K$ in  $\{\mathbb{Q}(\sqrt{-2}),\mathbb{Q}(\sqrt{-19})\}.$ We compute $X_0(15)(\mathbb{K})$ has rank 0 and torsion $\mathbb{Z}_2 \oplus \mathbb{Z}_4$.  The computations in \cite{New16} show that there are 4 non-cuspidal torsion points. Only two of them correspond to a cyclic 15-isogeny defined over extensions of $K$ of degree at least 2. More precisely, each of these non-cuspidal points gives rise to a torsion subgroup of order 15 defined over either $K(\sqrt{5})$ or $K(\sqrt{-15})$. Moreover, these torsion points produce four elliptic curves over $K$ (up to isomorphism) with torsion isomorphic to $\mathbb{Z}_{15}$ upon the base change to $K(\sqrt{5})$ or $K(\sqrt{-15}).$ 

\section*{ $N=21$ }
$X_0(21)(\mathbb{Q})$ has rank 0, torsion $\mathbb{Z}_2 \oplus \mathbb{Z}_4$ and 4 cusps. The torsion does not change when the base field $\mathbb{Q}$ is extended to any $K$ in $\mathcal{S}.$ As shown in Table \ref{table:2}, each of four non-cuspidal points corresponds to a cyclic 21-isogeny defined over an extension of $K$ of degree at least 3. So none of these rational points produce a cyclic 21-isogeny over a quadratic extension of $K.$

\begin{table}[h!]
\centering
\begin{tabular}{ p{2.5cm} p{4cm} p{5cm} p{2cm}  }
 
 Point & j(E) & E & $f_C$\\
 \hline
 (-1/4,1/8)  &  3375/2 &  [20/441, -16/27783]   &   (1,3,6)\\
 \hline
 (2,-1) &    -189613868625/128 & [-1915/36, -48383/324]   & (1,3,6)\\
 \hline
 (-1,2) & -1159088625/2097152 & [-505/192, -23053/6912] & (1,3,6)\\
 \hline
 (5,-13)   & -140625/8 & [-1600/147, -134144/9261] & (1,3,3,3)\\
 \hline
\end{tabular}
\caption{We use the model $y^2 + xy = x^3 - 4x - 1 $ for $X_0(21).$}
\label{table:2}
\end{table}

The rank of $X_0(21)$ only goes up in $K=\mathbb{Q}(\sqrt{-43}).$ Hence a search of points on $X_0(21)(K)$ is computationally infeasible to determine whether $E(\overline{K})$ can have a cyclic subgroup of order 21. Instead we will use the division polynomial method to prove the following.

\begin{prop}\label{propo}

Let $E$ be an elliptic curve over $K=\mathbb{Q}(\sqrt{-43})$ and $E^d$ runs through the quadratic twists of $E$. If $E(K) \simeq \mathbb{Z}_7$ then $E^d(K) \not\simeq \mathbb{Z}_3$ for all non-square $d \in K$.

\end{prop}

\begin{proof}

Assume $E(K) \simeq \mathbb{Z}_7$ and $E^d(K) \simeq \mathbb{Z}_3$ for some non-square $d \in K$. Then $E(K)$ has a point of order 7 and it has an additional $K$-rational subgroup $C$ of order 3 (arising from the twist). By Table 3 in \cite{Kub76}, we take the families of elliptic curves over $K$ with 7-torsion as
\begin{align}\label{three}
	 E_t : y^2 + (1-c)xy - by = x^3 - bx^2 
\end{align}
where $b = t^3 - t^2$ and $c= t^2 -t$ for some $t \neq 0,1$ in $K.$ We may assume $E$ is $E_{t_0}$ for some $t_0$.
 
Let $P$ be a point on $E$ generating the $K$-rational subgroup $C$. Then, $P$ has order 3 which is defined over $K(\sqrt{d})$. Write $x(P),$ for the $x$-coordinate of $P.$ We claim that $x(P)$ is in $K.$ Since $\langle P \rangle$ is a $K$-rational subgroup consisting only  $\{P, -P\}$ with $x(P) = x(-P),$ it follows that $x(P)^{\sigma} = x(P^{\sigma}) = x(P) $ hence $x(P)$ is invariant under the action of Gal$(\overline{K}/K)$. Now the pair $(E,P)$ corresponds to a point $(x(P),t_0)$ on the curve $C$ given by the equation $ \phi(x,t) = 0 $ where $\phi(x,t)$ is the third division polynomial of the generic $E_t$ :

$$ \phi(x,t) := x^4 + (\frac{1}{3}t^4 - 2t^3 + t^2 + \frac{2}{3}t + 1/3)x^3 + (t^5 - 2t^4 + t^2)x^2 + (t^6 - 2t^5 + t^4)x + (-\frac{1}{3}t^9 + t^8 - t^7 + \frac{1}{3}t^6) $$

It comes down to find the set $C(K)$ of $K$-rational points on $C$. By Magma, $C$ has singular points at $\{(0,0),(0,1)\}$ and $C$ is isomorphic to the hyperelliptic curve 
$$ \tilde C : y^2=f(u)=u^8-6u^6- 4u^5 + 11u^4 + 24u^3 + 22u^2 + 8u +1 $$
over $\mathbb{Q}$ 
outside of the singular points. By the choice of $t \neq 0,1$ in \eqref{three}, it is enough to find $\tilde C(K).$ Now $\tilde C$ is a hyperelliptic curve of genus 3 and its defining polynomial factors as: 
$$f(u) = (u^2 + u +1)(u^6 - u^5 -6u^4 + 3u^3 + 14u^2 + 7u +1)$$

Let $\tilde C^{(2)}$ be the symmetric square of $\tilde C$. The set $\tilde C^{(2)}(\mathbb{Q})$ consists of the equivalence classes of pairs $\{P,\overline{P}\}$ of a quadratic point and its Galois conjugate, as well as pairs $\{P,Q\}$ of rational points in $\tilde C.$ To study all quadratic points on a curve of genus greater than 2, one often studies the rational points on the symmetric square of the curve. Write $\tilde J$ for the Jacobian of $\tilde C.$ The map 
\begin{align*}
 \phi : \tilde C^{(2)} &\rightarrow \tilde J \\
\{P,Q\} & \mapsto [P+Q - \infty_{+} - \infty_{-}]
\end{align*}
is an injection outside of 0. The copy of $\mathbb{P}^1$ inside $\tilde C$ gets contracted to 0, under $\phi.$ So $\phi^{-1}(0)$ is isomorphic to $\mathbb{P}^1.$ We have 
$$ \tilde C^{(2)} (\mathbb{Q}) = \phi^{-1}(0)(\mathbb{Q}) \cup \phi^{-1}(\tilde J(\mathbb{Q}) \backslash \{0\}). $$ 
We first compute $\phi^{-1}(\tilde J(\mathbb{Q}) \backslash \{0\}).$ By $2$-descent computation in Magma, $\tilde J(\mathbb{Q})$ has Mordel-Weil rank 0. As $\tilde J$ has a good reduction at 5, the map $\tilde J(\mathbb{Q}) \rightarrow \tilde J(\mathbb{F}_{5})$ is injective. Since the group generated by the $2$-torsion point corresponding to the factorization of $f(u)$ on the right-hand side of the curve equation and the difference of the two points at infinity on $\tilde C$, which are $\infty_{+}=[1,1,0]$ and $\infty_{-}=[1,-1,0]$, surjects onto $\tilde J(\mathbb{F}_5) \simeq \mathbb{Z}_{2} \times \mathbb{Z}_{52}$, it is therefore equal to $\tilde J(\mathbb{Q}).$

Now consider a point $P' \in \tilde C(K)$ and write $\overline{P'}$ for its image under the nontrivial automorphism of $K.$ Then $P'+\overline{P'}$ is an effective divisor of degree 2 on $\tilde C,$ defined over $\mathbb{Q}.$ For the effective rational divisor $\infty_{+} + \infty_{-}$ of degree 2, the linear equivalence class of $P'+\overline{P'} -\infty_{+} - \infty_{-} $ is a rational point on $\tilde J.$ It follows that $P' + \overline{P'}$ is linearly equivalent to $D + 
	\infty_{+} + \infty_{-} $ for some divisor $D$ of degree 0 with $[D] \in \tilde J(\mathbb{Q}).$ We may now assume that $x(P') \notin \mathbb{Q}$. It remains to enumerate all 104 elements of $\tilde J(\mathbb{Q})$ for $[D] \neq 0$ in Magma and check if these divisors are of degree 2 with the points in the support defined over $K.$ 
	$$ \left(\frac{-1 \pm \sqrt{3}}{2}, 0 \right) , \left(\pm \sqrt{2}, \pm 4\sqrt{2}+5 \right), \left(\pm \sqrt{2}, \mp 4\sqrt{2} - 5\right), \left(\frac{-2 \pm \sqrt{2}}{2}, \frac{-5 \pm 4\sqrt{2}}{4}\right), $$
	
$$	\left(\frac{-2 \pm \sqrt{2}}{2}, \frac{5 \mp 4\sqrt{2}}{4}\right), \left(1 \pm \sqrt{2}, 11 \pm 8\sqrt{2}\right), \left(1\pm \sqrt{2}, -11 \mp 8\sqrt{2}\right) $$

None of these points is defined over the base field $K$ which implies that there are no exceptional $K$-rational points on $\tilde C.$ i.e. no $K$-points with $x$-coordinate not in $\mathbb{Q}.$

This leaves the case when $x(P') \in \mathbb{Q}$. There are $\mathbb{Q}$-points which are two points at infinity and four points with $x(P')=0$ or $1$. For all other such points, $P' + \overline{P'}$ is linearly equivalent to $\infty_{+} + \infty_{-}$, and so $[D] = 0$. Thus $P' = (u, \pm \sqrt{f(u)})$ with $u \in \mathbb{Q}$ and so  the non-exceptional $K$-rational points would have rational $u$ if $f(u) = -43v_0$ for some rational $v_0.$ However, there are no such points, since $f(u)$ is always positive. 
	Therefore, we have shown that 
	$$ \tilde C(K) = \tilde C(\mathbb{Q})=\{\infty_{+}, \infty_{-},(0,1),(0,-1),(1,1),(1,-1)\}.$$ 
	
Now the pair $(E,P)$ associated with $(x(P),t_0)$ on $C(K)$ corresponds to a point on $\tilde{C}(K)$, so to a rational point $(u,y)$ on the curve $\tilde{C}$. Then the preimage $(x(P),t_0)$ of $(u,v)$ is in fact a rational point on $C.$ However, if $E$ is defined over $\mathbb{Q}$ then a $K$-rational 3-isogeny $\langle P \rangle$ together with a $K$-rational 7-torsion on the elliptic curve $E/\mathbb{Q}$ will yield a subgroup of order 21 in $E(K(\sqrt{d}))$. But this is impossible by Theorem 2 in \cite{Fuj04}. 
    
\end{proof}

\begin{rem}
Proposition \ref{propo} in fact shows that there is no 21-isogeny in a quadratic extension of $K.$ Over $K=\mathbb{Q}(i)$,  one can find a similar result in \cite{Ej18} by using a different technique. 
\end{rem}

\section*{ $N=27$ }
The modular curve $X_0(27)$ is an elliptic curve with the model $y^2+y = x^3-7.$ Over $\mathbb{Q}$, $X_0(27)$ has rank 0 and torsion $\mathbb{Z}_3$ with 2 cusps. By Magma, the torsion does not change over $K$ for any $K$ in $S,$ and the rank only goes up in the extensions $K=\mathbb{Q}(\sqrt{-2})$ and $K=\mathbb{Q}(\sqrt{-11}).$ In case of positive rank, it is not efficient to use $K$-rational points on $X_0(27)$ to show if there exists a subgroup of order $27$ in a quadratic extention of $K.$ To achieve this goal, we applied the divison polynomial method as used in the previous case $N=21$ and obtained the following elliptic curves over $K$ with a subgroup of order $27$ defined over a quadratic extension of $K.$ 

\begin{rem}\label{rem4}
(1)  Let $K=\mathbb{Q}(\sqrt{-2})$ and $E/K$ be an elliptic curve given by 
$$ 
y^2 + \frac{1}{27}(-950w-619)xy + \frac{1}{243}(16720w-210862)y = x^3+\frac{1}{243}(16720w- 210862)x^2
$$
where $w=\sqrt{-2}.$ Magma computes that $E(K)_{\text{tor}} \simeq \mathbb{Z}_9$ and $E(K(\sqrt{-3}))_{\text{tor}} \simeq \mathbb{Z}_3 \oplus \mathbb{Z}_9.$ \\

(2) Let $K=\mathbb{Q}(\sqrt{-11})$ and $E/K$ be an elliptic curve given by
$$
y^2 + \frac{1}{27}(-2072w-4265)xy + \frac{1}{243}(-949568w + 377548)y = x^3 + \frac{1}{243}(-949568w + 377548)x^2
$$
where $w=\sqrt{-11}$. By Magma, $E(K)_{\text{tor}} \simeq \mathbb{Z}_9$ and $E(K(\sqrt{-3}))_{\text{tor}} \simeq \mathbb{Z}_3 \oplus \mathbb{Z}_9.$
\end{rem}

\section*{ $N=33$ }
The modular curve $X_0(33)$ is a hyperelliptic curve of genus 3 with the model
\begin{align*}
y^2 + (-x^4 - x^2 - 1)y = 2x^6 - 2x^5 + 11x^4
    - 10x^3 + 20x^2 - 11x + 8 
\end{align*}
which can be considered as a quadratic equation in $y$ as: $y^2 - g(x)y-h(x) = 0$ where $g(x)=x^4+x^2+1$ and $h(x)=2x^6 - 2x^5 + 11x^4- 10x^3 + 20x^2 - 11x + 8$. So its discriminant is
\begin{align}
    f_{33}(x)= x^8+10x^6-8x^5+47x^4-40x^3+82x^2-44x+33.
\end{align} 

Galbraith \cite[Table 4]{Glb96} derives $y^2=f_{33}(x)$ as a model for $X_0(33).$ Using Ogg's method we find that there are only 6 cusps, which are defined over $\mathbb{Q}.$
\begin{center}\
	
	\begin{tabular}{p{2cm} | p{1.25cm}| p{1.25cm}| p{1.25cm}| p{1.25cm}}
		
		d & 1 &  3 & 11 & 33  \\
		\hline
		$\phi(d,33/d)$ & 1 & 1 & 1 & 1  \\

	\end{tabular}
\end{center}\

Write $J_0(33)$ for the Jacobian of $X_0(33)$. As known $X_0(33)(K)$ can be put inside its 2-th symmetric product $X_0(33)^{(2)}(\mathbb{Q}),$ it is enough to determine $\mathbb{Q}$-points of $X_0(33)^{(2)}.$ Since $i=w_{11}$ is the hyperelliptic involution of $X_0(33)$, $X_0(33)^{(2)}$ injects into $J_0(33)$, except above $0$, via mapping the pair $\{P,Q\}$ to the equivalence class of $[P+Q - \infty_{+} - i(\infty_{+})]$ where $\infty_{+}= [1,1,0]$. It boils down to pull back the $\mathbb{Q}$-points of $J_0(33)$ under this injection. $J_0(33)(\mathbb{Q})$ has rank 0 by 2-descent computation in Magma with good reduction at $p=5,7.$ Moreover, $X_0(33)(\mathbb{Q})$ consists entirely of cusps, so one can compute the group structure of the cuspidal subgroup $C_{0}(33)$ to be  $\mathbb{Z}_{10} \oplus \mathbb{Z}_{10}$ which is generated by the difference of two points at infinity, $\infty_{+}=[1,1,0]$, $i(\infty_{+})=[1,0,0]$ and the difference of a rational point and one of the points at infinity. Also, we compute 
$$ J_0(33)(F_5) \simeq \mathbb{Z}_{10} \oplus \mathbb{Z}_{20}, \ \ \ \ \ J_0(33)(F_7) \simeq \mathbb{Z}_2 \oplus \mathbb{Z}_2  \oplus \mathbb{Z}_{10} \oplus \mathbb{Z}_{10}.$$

As $J_0(33)(\mathbb{Q})$ injects into both $J_0(33)(F_5)$ and $J_0(33)(F_7),$
we show that $J_0(33)(\mathbb{Q}) =C_{0}(33).$ By enumerating all the 99 elements of $J_0(33)(\mathbb{Q})\backslash \{0\}$ in Magma, we compute all the exceptional quadratic points as in Table \ref{table:3}. Each point corresponds to an isomorphism class of an elliptic curve over $K$ with a cyclic 33-isogeny defined over an extension of $K$ of at least degree 10.
   
It remains to determine the non-exceptional quadratic points on $X_0(33).$ Given $K=\mathbb{Q}(\sqrt{d})$ in $\mathcal{S},$ where $d \in \{-2,-7,-11,-19,-43,-67,-163\},$ the non-exceptional $K$-rational points would have rational $x$ with $f_{33}(x) =-dy_0^2$ for some rational $y_0.$ Since $f_{33}(x)$ is decomposed into the conjugate factors over $\mathbb{Q}(\sqrt{-11})$ as 
$(x^2-x+3)(q(x)+r(x)\sqrt{-11})(q(x)-r(x)\sqrt{-11})$
for some polynomials $q(x),r(x)$ with rational coefficients. One can observe that $x^2-x+3$ is positive for all $x$. Moreover $f_{33}(x)=0$ if and only if $q(x)^2 + 11r(x)^2 =0$ i.e. both $q(x)$ and $r(x)$ are zero which is impossible. Thus $f_{33}$ happens to be positive for all $x$ so it cannot be equal to $-dy_0^2$ for any rational $y_0.$ Therefore, there are no non-exceptional $K$-rational points on $X_0(33).$

\begin{table}[h!]
\centering
\begin{tabular}{ p{4.4cm}  p{11.04 cm}  p{1.2cm} }
 Point  &  Elliptic Curve & $f_C$\\
\hline 
$\left(\pm \sqrt{-2},\pm \sqrt{-2}+2 \right)$   &  $\big[\mp 360\sqrt{-2}+420, \mp 1008\sqrt{-2}-5600\big]$  &  (1,5,10)\\

$\left(\pm \sqrt{-2}, \mp \sqrt{-2}+1 \right)$ & $\big[\pm 360\sqrt{-2}+420, \pm 1008\sqrt{-2}-5600\big]$  &  (1,5,10) \\

 $\left(\pm \sqrt{-2}-1,  \pm 7\sqrt{-2}-2 \right)$ & $\big[\frac{1}{8}(\pm 1404\sqrt{-2}-2157),\frac{1}{16}(\pm 44171\sqrt{-2}+6156)\big]$ & (1,5,10) \\

$\left(\pm \sqrt{-2}-1,  \mp 5\sqrt{-2}-5 \right)$ & $\big[\frac{1}{8}(\mp 4356\sqrt{-2} + 363),\frac{1}{16}(\mp 64009\sqrt{-2} + 100188)\big]$ & (1,5,10)\\

 $\left(\frac{\pm \sqrt{-2}}{2},\frac{\pm 4\sqrt{-2}-5}{4} \right)$ & $\big[-\frac{8349}{2}, \frac{1}{2}(\mp 847\sqrt{-2}+209088)\big]$   &   (1,5,10)\\

 $\left(\frac{\pm \sqrt{-2}}{2},\mp \sqrt{-2}+2 \right)$  & $\big[\frac{1}{2}(\pm 5040\sqrt{-2}+28011),\frac{1}{2}(\pm 880397\sqrt{-2}-488304)\big]$ & (1,5,10)\\

 $\left(\frac{\pm \sqrt{-7}+1}{2},-1\right)$  & [$\frac{1}{128}(\pm 296769\sqrt{-7}+323829),\frac{1}{512}(\pm 7944695\sqrt{-7} + 104422307)\big]$   & (1,5,10)\\

$\left(\frac{\pm \sqrt{-7}+1}{2},\mp \sqrt{-7}+1\right)$ &$ \big[\frac{1}{128}(\mp 27591\sqrt{-7}-140211), \frac{1}{512}(\pm 2704139\sqrt{-7}+4638407)\big]$ & (1,5,10) \\

$\left(\frac{\pm 3\sqrt{-7}+1}{4}, \frac{\mp 9\sqrt{-7}+33}{4}\right)$ & $ \big[ \frac{1}{8}(\mp 2553\sqrt{-7} + 4365), \frac{1}{8}(\pm 2645\sqrt{-7} + 96359) \big] $ & (1,5,10) \\

$\left(\frac{\pm 3\sqrt{-7}+1}{4}, \frac{\mp 9\sqrt{-7}+93}{32}\right)$ & $\big[\frac{1}{8}(\mp 33\sqrt{-7} - 2475),\frac{1}{8}(\pm 11033\sqrt{-7} + 6611)\big]$ & (1,5,5,5) \\

$\left(\frac{1\pm \sqrt{-11}}{2},\mp \sqrt{-11} +1 \right)$ & $\big[-1056, -13552\big] $ & (1,5,10) \\

\hline
\end{tabular}
\caption{}
\label{table:3}
\end{table}
In summary, we have shown that the set of $K$-rational points of $X_0(33)$ consists only $\mathbb{Q}$-points.
$$
X_0(33)(K)=X_0(33)(\mathbb{Q})=\{\infty_{+},i(\infty_{+}),(1,-3),(1,6)\}.
$$

Since these are all cusps, we prove that there are no elliptic curves over $K$ with a cyclic 33-isogeny in a quadratic extension of $K.$

\section*{ $N=35$ }
The modular curve $X_0(35)$ is a hyperelliptic of genus 3. Galbraith \cite[Table 4]{Glb96} gives the following model
\begin{align*}
 X_0(35) : y^2 &= x^8-4x^7-6x^6 -4x^5 -9x^4+4x^3-6x^2+4x+1 \\ 
          &= (x^2+x-1)(x^6-5x^5-9x^3-5x-1) 
\end{align*}           
    
By Ogg's theorem, we find that the set of cusps entirely consists of $\mathbb{Q}$-points.   
\begin{center}\
	
	\begin{tabular}{p{2cm} | p{1.25cm}| p{1.25cm}| p{1.25cm}| p{1.25cm} }
		
		d & 1 & 5 & 7 & 35 \\
		\hline
		$\phi(d,35/d)$ & 1 & 1 & 1 & 1  \\

	\end{tabular}
\end{center}\       
          
Given $K \in \mathcal{S},$ 
we study the group $J_0(35)(\mathbb{Q})$ where $J_0(35)$ is the Jacobian of $X_0(35)$. $J_0(35)(\mathbb{Q})$ has rank 0 by 2-descent computation in Magma. As $J_0(35)$ has good reduction at 3, the reduction map $J_{0}(35)(\mathbb{Q}) \rightarrow J_0(35)(\mathbb{F}_3)\simeq \mathbb{Z}_{2} \oplus \mathbb{Z}_{24}$ is injective. We find 
$$ \mathbb{Z}_2 \oplus \mathbb{Z}_{24} = \langle  [3(0,-1)-3\infty_{+}], [\infty_{-} - \infty_{+}] \rangle $$
which therefore implies $J_0(35)(\mathbb{Q}) =\mathbb{Z}_2 \oplus \mathbb{Z}_{24}.$
Let $X_0(35)^{(2)}$ be the symmetric square of $X_0(35)$.  
Since finding quadratic points on $X_0(35)$ amounts to finding the rational points on the symmetric square $X_0(35)^{(2)}$, we consider the following map 
\begin{align*}
 \phi : X_0(35)^{(2)} &\rightarrow J_0(35) \\
\{P,Q\} & \mapsto [P+Q - \infty_{+} - \infty_{-}]
\end{align*}

Let $P$ be a quadratic point in $X_0(35)(K)$ with x-coordinate not in $\mathbb{Q}$ and let $\overline{P}$  denote its Galois conjugate. The linear equivalence class of the degree 0 divisor $P+\overline{P}-\infty_{+} -\infty_{-}$ gives rise to a nonzero $\mathbb{Q}$-rational point on $J_0(35).$ Magma represents all 47 elements of $J_0(35)(\mathbb{Q})\backslash \{0\}$ as effective divisors minus a multiple of one of the points at infinity. We check that none of the points is of this form except $(\frac{-1 \pm \sqrt{5}}{2},0)$. Therefore, there is no exceptional $K$-rational points on $X_0(35).$

If $x(P) \in \mathbb{Q},$ there are  $\mathbb{Q}$-points which are two points at infinity and two points with $x(P)=0$; for all other such points, the Galois conjugate $\overline{P}$ must be the image of $P$ under the hyperelliptic involution  $w_{35}$
of  $X_0(35)$ $\text{given by}$ $y \mapsto -y$. Thus $y(P)$ $\text{must be}$  $\sqrt{d}$ times a rational number where $K=\mathbb{Q}(\sqrt{d})$ with $d$ a non-square. Thereby, $P$ gives rise to a $\mathbb{Q}$-rational point on the quadratic twist $X_0^d(35)$ of $X_0(35)$ by $d$. The quadratic fields of interest are in fact of the form $\mathbb{Q}(\sqrt{d})$ where $d=-p$ with $p$ prime in the set $\{2,7,11,19,43,67,163\}.$ However, by Corollary 3.5 in \cite{Ozman} the quadratic twist $X_0^{d}(35)$ of $X_0(35)$ does not even have $\mathbb{Q}_p$-rational points, let alone $\mathbb{Q}$-points which shows that
$$ X_0(35)(K) = X_0(35)(\mathbb{Q}) = \{ \infty_+,\infty_{-}, (0,1),(0,-1)\} .$$
Since $\mathbb{Q}$-points are cuspidal, we conclude that there are no elliptic curves over $K$ with a $35$-isogeny defined over a quadratic extension of $K.$

\section*{ $N=45$ }
$X_0(45)$ is  a non-hyperelliptic curve of genus $3$ with a projective model
$$ x^2y^2 +x^3z - y^3z - xyz^2 + 5z^4 $$

The quadratic points on $X_0(45)$ are determined in  \cite{OS18} by using a slightly different model. There are only two noncuspidal quadratic points defined over an imaginary quadratic field of class number one. They are $P=( \frac{1-w}{2},-1,1) \  \text{and} \ Q=(1,\frac{w-1}{2},1)$ where $w=\sqrt{-11}.$
One can check in Magma that these points do not produce a 45-isogeny in a quadratic extension of $\mathbb{Q}(\sqrt{-11}).$

\section*{ $N=55$ }
$X_0(55)$ is a non-hyperelliptic curve of genus 5. The quadratic points of $X_0(55)$ are listed in \cite{OS18}. One can verify by Magma that the ones defined over an imaginary quadratif field $K$ of class number one do not
yield a 55-isogeny in a quadratic extension of $K.$


\section*{$ N=77$ }
The modular curve $X_0(77)$ is a non-hypelliptic curve of genus 7. However, $X_0(77)$ is not in the database of the Small Modular Curves package in Magma. Instead, we compute the model by using the code written by Ozman and Siksek \cite{OS18}.

We start with computing the torsion subgroup of the Jacobian. Write $J_0(77)$ for the Jacobian of $X_0(77).$ Note that all the cusps of $X_0(77)$ are rationals. Let $C_0(77)$ be the cuspidal subgroup of $J_0(77)(\mathbb{Q})$ generated by the classes of differences of the cusps. 
In the next lemma, we show that the cuspidal subgroup of $X_0(77)$ in fact satisfies the equality $C_0(N)=J_0(N)(\mathbb{Q})_{\text{tor}}.$ 

\begin{lem}
	The torsion subgroup of $J_0(77)(\mathbb{Q})$ is equal to its cuspidal subgroup $C_0(77)$ which is isomorphic to $\mathbb{Z}_{10} \oplus \mathbb{Z}_{60}.$ 
\end{lem}

\begin{proof}
	In Magma, one computes the group structure of $C_0(77)$ to be $\mathbb{Z}_{10} \oplus \mathbb{Z}_{60}.$ As 3 and 13 are primes of good reduction for $X_0(77)$, we compute that 
	$$
	J_0(77)(\mathbb{F}_3) \simeq \mathbb{Z}_{10} \oplus \mathbb{Z}_{420}
	$$
	$$ J_0(77)(\mathbb{F}_{13}) \simeq (\mathbb{Z}_2)^3 \oplus (\mathbb{Z}_{30})^2 \oplus \mathbb{Z}_{7380} 
	$$
Since $J_0(77)(\mathbb{Q})_{\text{tor}}$ injects into both of these groups, then  $[J_0(77)(\mathbb{Q})_{\text{tor}} : C_0(77)]$ divides 7 and $2^3.3^3.5.41$ respectively. But these are coprime. Hence we show $J_0(77)(\mathbb{Q})_{\text{tor}}=C_0(77).$

\end{proof}

Generally speaking, given a divisor $N_1$ of level $N$ with $(N,N/N_1)=1$, the Atkin-Lehner involution $w_{N_1}$ acts as a permutation on the set of cusps of $X_0(N).$ Let $B_0(77)$ denote the group generated by the Atkin-Lehner involutions of $X_0(77).$ As shown in \cite{KM88Second}, the automorphism group Aut$(X_0(77)$) of $X_0(77)$ is $B_0(77)$ and we have $ B_0(77)= \left<w_7\right> \oplus \left<w_{11}\right>.$ One can easily check that $B_0(77)$ acts transitively on the set of cusps. Let $\Gamma =\left<w_{77}\right>$ be the subgroup of Aut($X_0(77)$) and consider the quotient curve $X_0(77)^{+}=X_0(77)/\Gamma$ with the projection $\rho : X_0(77) \rightarrow X_0(77)^{+}.$ This map has degree $deg\rho=\#\Gamma.$ We now denote the Jacobian of $X_0(77)^{+}$ by $J_0^{+}(77)$ to avoid the confusion with the Jacobian $J_0(77)$ of $X_0(77).$

\begin{lem}\label{rank}
	\normalfont{rk}($J_0(77)(\mathbb{Q})) =$ \normalfont{rk}($J_0^{+}(77)(\mathbb{Q}))$.
\end{lem}

\begin{proof}
	We use how $J_0(77)$ is isogenous to a product of abelian varieties associated to newforms, following \cite[p. 234]{K08}. More presicely, we have
	\begin{align}
	J_0(77) \sim A_{f_1} \times A_{f_2} \times A_{f_3} \times A_{f_4} \times A_{f_5}^{2} 
	\end{align}
	where $\sim$ denotes isogeny over $\mathbb{Q}$ and $A_{f_i}$'s are simple modular abelian varieties corresponding to Galois orbits $f_i$ of newforms of weight 2 and level 77. It follows by (5) that rk($J_0(X)) = \sum_{i=1}^{5} m_iA_{f_i}$ where $m_i$ denotes multiplicity. We also found that $J_0^{+}(77) \sim A_{f_1} \times A_{f_5}.$ With the modular abelian varieties package in Magma, we checked that the L-function $L(A_{f_i},1)$ is non-zero for the eigenforms $f_i$ in (5) not conjugate to $f_1$. By the theorem of Kolyvagin and Logachev \cite{KL89}, this implies that the corresponding $A_{f_i}$ has analytic rank 0 so is the algebraic rank over $\mathbb{Q}$. Therefore, rk($J_0(77)(\mathbb{Q}))=A_{f_1}=$rk($J_0^{+}(77)(\mathbb{Q})).$
	
\end{proof}

We now need information about the Mordell-Weil group of $J_0^{+}(77)(\mathbb{Q}).$ Note that $X_0(77)^{+}$ is a hyperelliptic curve of genus 2 by \cite{Yuj95}. Using the algorithm in \cite{Sto01} implemented in Magma as 2-descent on Jacobians of hyperelliptic curves, we compute that $J_0^{+}(77)(\mathbb{Q})$ has Mordell-Weil rank 1 as well as its generators. 

We would like to describe all quadratic points on $X_0(77).$ One way of finding quadratic points on $X_0(77)$ is by pulling back of rational points on $X_0(77)^{+}$ under $\rho.$ We will call points on $X_0(77)$ not arising as pullback of a rational point on $X_0(77)^+$ $exceptional.$ Our aim is to prove the following:
\begin{thm}\label{77thm}
	There are no exceptional quadratic points on $X_0(77).$
\end{thm}

For this purpose, we use the Chabauty method for symmetric powers of curves developed by Siksek in \cite{Sek09} combining with a Mordell-Weil sieve. In order to apply this, we slightly change how we view quadratic points similarly as in the previous cases of $N.$ For a given smooth curve $X$ over $\mathbb{Q},$  we view a pair of quadratic point $P$ on $X$ and its Galois conjugate $\{P,\overline{P}\}$ as a rational point on the symmetric product $X^{(2)}.$ The pair $\{P,\overline{P}\}$ is represented as an equivalence class of degree 2 effective divisor $D:=P+\overline{P}$ in $X^{(2)}(\mathbb{Q}).$ 
In the particular case $X=X_0(77)$, if we pull back $P \in X_0(77)^{+}(\mathbb{Q})$, $\rho^{*}(P)$ is a pair of quadratic points on $X_0(77).$ We call such quadratic points $non$-$exceptional$. Hence we can rephrase Theorem \ref{77thm} as $X_0(77)^{(2)}(\mathbb{Q}) = \rho^{*}(X_0(77)^{+}(\mathbb{Q})).$ 

Morevoer, $X_0(77)^{+}(\mathbb{Q})$ can be completely determined by explicit Chabauty command built in Magma. Once we know the full list of points in $X_0(77)^{+}(\mathbb{Q})$, combining with Theorem \ref{77thm}, this gives us a complete list of all quadratic points on $X_0(77).$ In the light of these facts, we found that $X_0(77)^{+}$ has precisely six rational points and so by the theorem above, $X_0(77)$ has precisely six pairs of quadratic points. We are able to list these quadratic points explicitly by using the equations for the degree 2 map from $X_0(77) \rightarrow X_0(77)^{+}.$ These equations come for free while computing the model of $X_0(77).$

\begin{table}[h!]
	\centering
	\begin{tabular}{ | p{1.2cm}|p{2.3cm}| }
		\hline
		Point & Coordinates  \\
		\hline
		$Q_1$  &  [1:-1:0]   \\
		\hline
		$Q_2$ &  [1:1:0]   \\
		\hline
		$Q_3$ & [-1:-2:1]  \\
		\hline
		$Q_4$ & [-1:2:1]  \\
		\hline
		$Q_5$ & [0:-1:1]  \\
		\hline
		$Q_6$ & [0:1:1]  \\
		\hline
	\end{tabular}
	\caption{The rational points on $X_0(77)^{+}$}
	\label{table:3.2}
\end{table}

$\bold{Chabauty\ Step}:$ Let $X$ be a smooth projective curve over $\mathbb{Q}$ of genus $\geq 2$ with Jacobian $J(X)$ of rank $r_X.$ We suppose that there exists an involution $w$ of $X$ defined over $\mathbb{Q}$ such that the quotient map $X \rightarrow X/\left<w\right>$ has degree 2. Define $C:=X/\left< w \right>$ and denote by $g_C$ for the genus of $C$ and $r_C$ for the rank of the Jacobian of $C$. For $p$ a prime of good reduction for both $X$ and $C,$ Coleman gives the following bilinear pairing
\begin{align}\label{map}
\Omega_{X_{/\mathbb{Q}_p}} \times J(\mathbb{Q}_p) \rightarrow \mathbb{Q}_{p}, \ \ \ \left(\omega, \left[\sum_{i} (P_i- Q_i) \right] \right) \mapsto \sum_{i} \int_{Q_i}^{P_i} \omega 
\end{align}
where $\Omega_{X_{/\mathbb{Q}_p}}$ is the space of regular differentials on the curve $X_{/\mathbb{Q}_p}.$ The kernel on the left is 0 and on the right is $J(\mathbb{Q}_p)_{\text{tor}}.$ 
Define $V$ to be the annihilator of $J(\mathbb{Q})$ with respect to \ref{map}. 
As being an algebraic curve over $\mathbb{Q}$, $X$ has a minimal regular model over $\mathbb{Z}_{p}$ which we denote by $\mathcal{X}.$ The $\mathbb{Z}_p$-scheme $\mathcal{X}$ is smooth when $\mathcal{X}_{\mathbb{F}_p}$ is non-singular. We consider  $\mathcal{V} := V \cap \Omega_{\mathcal{X}_{/\mathbb{Z}_p}}$ and $\widetilde{V}$ denote its reduction modulo $p$. The symmetric Chabauty differs from the usual Chabauty-Coleman method in the way that the annihilating differentials must also satisfy Tr$(\omega)=0$ where Tr: $\Omega_{X/\mathbb{Q}_{p}} \rightarrow \Omega_{C/\mathbb{Q}_{p}}$ is the trace operator. More precisely, we define $\mathcal{V}_0=\mathcal{V} \cap$ ker(Tr) and let $\widetilde{V_0}$ be its image under the reduction modulo $p.$
Note that $mod\ p\ residue\ disc$ of a point $Q \in X^{(2)}(\mathbb{Q})$ consists of points $P \in  X^{(2)}(\mathbb{Q})$ such that $\widetilde{P} \equiv \widetilde{Q}$ where $\sim$ denotes the reduction mod $p.$

Now let $\omega_1,...,\omega_r$ be a basis for $\widetilde{V}$ and let $Q$ be an element of $X^{(2)}(\mathbb{Q}).$ If $Q = Q_1 + Q_1$ as an effective degree 2 divisor, we say $Q_1$ has multiplicity $d_1=2.$ Otherwise we have $Q=Q_1 + Q_2$ where $Q_1,Q_2$ are distinct with multiplicty $d_j=1.$ Fix an extension $v$ to $L=\mathbb{Q}(Q_1,Q_2)$ above $p.$ Let $t_{\widetilde{Q_j}}$ be a uniformiser at $Q_j$ for each $j \in \{1,2\}.$ We can expand each $\omega_{i}$ as a formal power series at $Q_j$ as
$ \omega_i=\sum_{k=0}^{\infty} a_k(\omega_{i},t_{\widetilde{Q_j}})t_{\widetilde{Q_j}}^kdt_{\widetilde{Q_j}} $
We then define for a positive integer $m$

$$ \boldsymbol{v}(\omega_{i},t_{\widetilde{Q_j}},m)= \left(a_0(\omega_{i},t_{\widetilde{Q_j}}), \frac{a_1(\omega_{i},t_{\widetilde{Q_j}})}{2},...,\frac{a_{m-1}(\omega_{i},t_{\widetilde{Q_j}})}{m} \right)$$

Let $\widetilde{\mathcal{A}}$ be the following $ r \times 2$ matrix:

\begin{equation*}
\widetilde{\mathcal{A}} = 
\begin{pmatrix}
\boldsymbol{v}(\omega_{1},t_{\widetilde{Q_1}},d_1) & \boldsymbol{v}(\omega_{1},t_{\widetilde{Q_2}},d_2)  \\
\boldsymbol{v}(\omega_{2},t_{\widetilde{Q_1}},d_1) & \boldsymbol{v}(\omega_{2},t_{\widetilde{Q_2}},d_2) \\
\vdots  & \vdots   \\
\boldsymbol{v}(\omega_{r},t_{\widetilde{Q_1}},d_1) & \boldsymbol{v}(\omega_{r},t_{\widetilde{Q_2}},d_2)
\end{pmatrix}
\end{equation*}

Note that when $Q_1=Q_2$, we have $d_1=2$ with the assumption that $p > 2.$

\begin{thm}[Symmetric Chabauty, Siksek]\label{sym}
	 Let $Q=\{Q_1,Q_2\} \in X^{(2)}(\mathbb{Q}).$ For $p$ a prime of good reduction for $X,$ pick $\omega_1,...,\omega_r$  as a basis for $\widetilde{V}$ as above. Suppose $p >3$ when $[\mathbb{F}_p(\widetilde{Q_1}),\mathbb{F}_p]=1.$ If rank($\widetilde{\mathcal{A}})=2$ then $\boldsymbol{\mathcal{Q}}$ is alone in its mod $p$ residue class belonging to $X^{(2)}(\mathbb{Q}).$
\end{thm}

\begin{proof}
	This is the same statement as in \cite[Theorem 2.1]{Box20} which is a variant of the special case of Theorem 1 in \cite{Sek09}.
\end{proof}

We particularly make use of Theorem \ref{sym} for the points in $X_{0}(77)^{(2)}(\mathbb{Q}) \setminus \rho^{*}(X_{0}(77)^{+}(\mathbb{Q})).$ To deal with the points in $\rho^{*}(X_{0}(77)^{+}(\mathbb{Q}))$ we need to argue in a slightly different way using the relative version of Chabauty. The relative Chabauty differs from the symmetric Chabauty in the way that the annihilating differentials must satisfy Tr$(\omega)=0$ where Tr is the trace map: Tr: $\Omega_{X/\mathbb{Q}_{p}} \rightarrow \Omega_{C/\mathbb{Q}_{p}}.$ More precisely, we define $V_0=\mathcal{V}\ \cap$ ker(Tr) and let $\widetilde{V}_0$ be its image reduction modulo $p.$

\begin{thm}[Relative Symmetric Chabauty, Siksek]\label{relative}
	Let $Q=\{Q_1,Q_2\} \in \rho^*(C(\mathbb{Q})) \subseteq X^{(2)}(\mathbb{Q})$ be a non-exceptional point. Let $p$ be a prime of good reduction for both $X$ and $C$, morever assume that $p>3$ when $[\mathbb{F}_p(\widetilde{Q}_1),\mathbb{F}_p]=1.$ Let $t_{\widetilde{Q_1}}$ be a uniformiser at $Q_1$ and $\omega_1,...,\omega_s$ be a basis for $\widetilde{V_0}.$ If for some $i \in\{1,...s\},$
	$$ \frac{\omega_i}{dt_{\widetilde{Q_1}}}  \mid_{t_{\widetilde{Q_1}}=0} \neq 0 $$
	Then every element of $X^{(2)}(\mathbb{Q})$ in the residue class of $Q$ is in fact belonging to $\rho^{*}(C(\mathbb{Q})).$
\end{thm}

\begin{proof}
	This is a modification of special case of \cite[Theorem 2]{Sek09}. 
\end{proof}
In order to apply Theorem \ref{relative} to the case $X=X_0(77)$ and $C=X_0(77)^{+},$ we need to compute the image $\widetilde{V}_0$ of $\mathcal{V}_0= \mathcal{V}\ \cap$ ker(Tr: $\Omega_{X/\mathbb{Q}_p} \rightarrow  \Omega_{C/\mathbb{Z}_p})$ under the reduction mod $p$ where both $X$ and $C$  have a good reduction at $p.$ For this purpose, we need the following lemma from \cite{Box20}.

\begin{lem}\label{generallem}
	Let $X$ be a smooth projective curve with minimal proper regular model $\mathcal{X}/\mathbb{Z}_p$ and denote $\widetilde{X}=\mathcal{X}_{\mathbb{F}_p}.$ Then the reduction map $\Omega_{\mathcal{X}/\mathbb{Z}_p} \rightarrow \Omega_{\widetilde{X}/\mathbb{F}_p}$ is surjective.
\end{lem}

\begin{prop}
	Let $V$ be the image of $1-w_{77}^{*} : \Omega_{X/\mathbb{Q}} \rightarrow \Omega_{X/\mathbb{Q}}.$ Assume $r_{X}=r_{C}.$ For $p$ a good reduction for both $X$ and $C,$ we have 
$\widetilde{V}_0= (1-\widetilde{w}_{77}^{*})(\Omega_{\widetilde{X}})$ where $\sim$ denotes the reduction $mod\ p.$ 
\end{prop}

\begin{proof}
	Since $w_{77}$ is an involution,  $(1-w_{77}^{*})(\Omega_{X})=$ker($1+w_{77}^{*})$ which is by the definition the kernel of the trace map.
	Next we show that all the differentials in $V$  indeed annihilate $J(X)(\mathbb{Q})$ with respect to \eqref{map} i.e. $\int_{0}^{D} \omega =0\ for\ all\ \omega \in V \text{and} \ D \in J(X)(\mathbb{Q}).$ Since $r_{X}=r_C$ by Lemma \ref{rank}, then the index $N=[J(X)(\mathbb{Q}):\rho^{*}(J(C)(\mathbb{Q}))]$ is finite. Let $D \in J(X)(\mathbb{Q}).$ As $N.D \in \rho^{*}(J(C)(\mathbb{Q})),$ there exists $\Delta \in J(C)(\mathbb{Q}) $ such that $N.D=\rho^{*}(\Delta).$ By the properties of Coleman integration, we have
	$$ \int_{0}^{D} \omega = \frac{1}{N} \int_{0}^{N.D} \omega = \frac{1}{N} \int_{0}^{\rho^{*}(\Delta)} \omega= \frac{1}{N} \int_{0}^{D} \text{Tr}(\omega)=0$$
	for every $\omega \in (1-w_{77}^{*})(\Omega_{X}).$ 
	Let $\widetilde{W}$ be the image of $1-\widetilde{w}_{77}^{*}.$ We want to show that $\widetilde{W}=\widetilde{V}.$ Since $\omega \in$ ker($1-\widetilde{w}_{77}^{*})$ then $\widetilde{\omega} \in$ ker$(1-\widetilde{w}_{77}^{*}),$ we have $\widetilde{V} \subseteq \widetilde{W}.$ Now consider $\widetilde{\omega} \in \widetilde{W}$ and $\widetilde{v}$ such that $\widetilde{\omega} = (1-w_{77}^{*})(\widetilde{v}).$ By Lemma \ref{generallem}, $\widetilde{v}$ lifts to some $v$ in $\Omega_{\mathcal{X}/\mathbb{Z}_p}.$ Then $(1-w_{77}^{*})(v) \in \mathcal{V}$ is reduced to $\widetilde{\omega},$ as claimed.
\end{proof}

$\bold{The\ Mordell-Weil\ sieving\ step:}$ 
We adapted here the sieve used in \cite{Sek09} to apply to the case $X=X_0(77).$ The goal is to sieve for potential unknown quadratic points and obtain a contradiction. To start with, we need to choose primes $p_1,...,p_k.$ A detailed guideline on how to choose primes and in which order can be found in \cite{BS10}. In order to make the sieve work, we further need the following info:
\begin{itemize}
	\item A finite index subgroup $G \subseteq J(X)(\mathbb{Q})$ with generators $D_1,...D_n$.
	\item A positive integer $I$ with $I.J(X)(\mathbb{Q}) \subseteq G.$
	\item A non-empty set $\mathcal{L} \subseteq X^{(2)}(\mathbb{Q})$ of known non-exceptional quadratic points.
\end{itemize}

For the moment, we let $p:=p_i$ for some $i.$ Define $ \psi : \mathbb{Z}^n \rightarrow G \subseteq J(X)(\mathbb{Q})$ by $\psi(a_1,...,a_n) =\sum_i^{n}a_iD_i.$ Write $\iota: X^{(2)}(\mathbb{Q}) \rightarrow G$ by $D \mapsto I.[D-\infty]$ where $\infty$ is a rational degree 2 divisor and  let $\imath_{p}$ be the Abel-Jacobi map on $X^{(2)}(\mathbb{F}_p)$ with the basepoint $\widetilde{\infty}.$ Then we define $\psi_{p}$ to make the following diagram commute
\begin{center}
		\begin{tikzpicture}[>=stealth,->,shorten >=2pt,looseness=.5,auto]
		\matrix (M)[matrix of math nodes,row sep=1cm,column sep=16mm]{
		L &	X^{(2)}(\mathbb{Q})   & G  & \mathbb{Z}^n\\
		  &	\widetilde{X}^{(2)}(\mathbb{F}_p) & J(\mathbb{F}_p) &  \\
		};
		\draw(M-1-1)--node{$i$}(M-1-2);
		\draw(M-1-2)--node{$\iota$}(M-1-3);
		\draw(M-1-4)--node[above]{$\psi$}(M-1-3);
		\draw(M-1-2)--node[left]{$\sim$}(M-2-2);
		\draw(M-1-3)--node[left]{$\sim$}(M-2-3);
		\draw(M-2-2)--node{$\iota_p$}(M-2-3);
		\draw(M-1-4)--node{$\psi_p$}(M-2-3);
		\end{tikzpicture}
\end{center}

Let $P \in X^{(2)}(\mathbb{Q})\setminus \mathcal{L}$ be a (hypothetical) unknown point.  We will argue that such $P$ cannot exist. Since the diagram commutes, note that $\widetilde{\iota(P)}=\iota_p(\widetilde{P})  \in \text{Im}(\psi_p).$  Define the set
$$ H_p :=\{S \in \widetilde{X}^{(2)}(\mathbb{F}_p) : \iota_{p}(S) \in \text{Im}(\psi_p) \}$$
and $M_p \subseteq H_p$ as the set of points $S$ satisfying one of the following conditions
\begin{itemize}
	\item $S\notin \widetilde{L}$ 
	\item $S=\widetilde{R}$ for some $R \in L$ not satisfying the conditions of Theorem \ref{sym}. 
	\item $S=\widetilde{R}$ for some $R \in L \cap \rho^{*}(C(\mathbb{Q}))$ not satisfying the conditions of Thoerem \ref{relative} in the case of degree 2 map $X \rightarrow C.$
\end{itemize}

Notice that by the choice of our unknown point $P$, $\widetilde{P}$ is in $M_p.$  We use the Chabauty step to eliminate as many of the points in $M_p$ as possible. The more points that pass the Chabauty step, the better since this enables us to limit more possibilites for $\widetilde{P}.$ 
Since $\iota(M_p) \subseteq \psi_p(\mathbb{Z}^n),$ we form the following $\text{ker}(\psi_p)$-cosets 
$$ W_p := \psi_p^{-1}(\iota(M_p)).$$

If $\iota(P)=a_1D_1+...+a_nD_n$ then $(a_1,...,a_n) \in w + \text{ker}(\psi_p)$ for some $w+\text{ker}(\psi_p)$ in $W_p.$ Hence the set of cosets $W_p$ will carry the information of the list of possibilities for $\widetilde{P}.$
The beauty of the sieve lies in the observation that $W_p$ can be compared for different values of $p.$ This means that if we choose a different prime of good reduction, namely $p',$ then the preimage  $\psi_{p'}^{-1}(\iota(P))=(a_1,...,a_n) \in w'+\text{ker}(\psi_{p'})$ for some $w' + \text{ker}(\psi_{p'})$ in $W_{p'}$ as above. So $(a_1,...,a_n) \in W_{p} \cap W_{p'}.$ Then we repeat this process for each $p$ in the list  $p_1,...,p_k$ of primes of good reduction. We conclude the following.

\begin{prop}(Sieving Principle)\label{sieve}
	If $ \bigcap_{i=1}^{k} W_p = \emptyset,$
then $X^{(2)}(\mathbb{Q})=L.$	
	
\end{prop}

\section*{Magma Output} 

The model obtained for the modular curve $X_0(77)$ as a curve in $\mathbb{P}^6:$

$x_1x_3 - x_2^2 + 21x_2x_7 - 2x_3x_4 - 4x_3x_7 + 6x_4^2 - 19x_4x_5 + 6x_4x_6 - 21x_4x_7 + 13x_5^2 - 5x_5x_6 - 8x_5x_7 - 2x_6x_7 + 13x_7^2=0, $ 

$x_1x_4 - x_2x_3 - x_2x_7 - 2x_3x_4 + 47x_3x_7 + 2x_4^2 + 8x_4x_5 - 10x_4x_6 + 15x_4x_7 - 26x_5^2 + 25x_5x_6 - 37x_5x_7 - 3x_6^2 + 3x_6x_7 - 73x_7^2=0, $ 

$ x_1x_5 - 11x_2x_7 - x_3^2 - 3x_3x_4 + 80x_3x_7 + 21x_4x_5 - 17x_4x_6 + 33x_4x_7 - 48x_5^2 + 42x_5x_6 - 56x_5x_7 - 5x_6^2 + 5x_6x_7 - 126x_7^2=0, $

$x_1x_6 - 18x_2x_7 - 5x_3x_4 + 99x_3x_7 + 29x_4x_5 - 22x_4x_6 + 47x_4x_7 - 62x_5^2 + 54x_5x_6 - 70x_5x_7 - 7x_6^2 + 8x_6x_7 - 161x_7^2=0, $

$x_1x_7 - 2x_2x_7 - x_3x_7 - x_4^2 + 2x_4x_5 + 2x_4x_7 - x_5^2 + x_5x_7 - 2x_7^2=0, $

$x_2x_4 - 9x_2x_7 - x_3^2 + 34x_3x_7 - 2x_4^2 + 13x_4x_5 - 9x_4x_6 + 18x_4x_7 - 23x_5^2 + 19x_5x_6 - 21x_5x_7 - 2x_6^2 + 2x_6x_7 - 54x_7^2=0, $

$x_2x_5 - 16x_2x_7 - x_3x_4 + 30x_3x_7 - 3x_4^2 + 18x_4x_5 - 9x_4x_6 + 25x_4x_7 - 25x_5^2 + 18x_5x_6 - 17x_5x_7 - 2x_6^2 + 4x_6x_7 - 54x7^2=0, $

$ x_2x_6 - 23x_2x_7 + 28x_3x_7 - 5x_4^2 + 23x_4x_5 - 10x_4x_6 + 29x_4x_7 - 26x_5^2 + 17x_5x_6 - 11x_5x_7 - 2x_6^2 + 5x_6x_7 - 55x_7^2=0, $

$ x_3x_5 - 16x_3x_7 - x_4^2 - 3x_4x_5 + 4x_4x_6 - 5x_4x_7 + 9x_5^2 - 9x_5x_6 + 13x_5x_7 + x_6^2 - x_6x_7 + 24x_7^2=0, $

$ x_3x_6 - 23x_3x_7 - 5x_4x_5 + 5x_4x_6 - 7x_4x_7 + 13x_5^2 - 12x_5x_6 + 17x_5x_7 + x_6^2 + 35x_7^2=0 $

Genus of $X_0(77)$ : 7.

Cusps:$(1,0,0,0,0,0,0),(1/2,1/2,1/2,1/2,1/2,1,0),(5,3,3,2,2,2,1),(5,3,3,2,2,9,1). $

$X^{+}=X_0(77)^{+}$ : genus 2 hyperelliptic curve $y^2 = x^6 + 2x^5 + 7x^4 + 8x^3 + 7x^2 + 2x + 1.$

Group Structure of $J(X^{+})(\mathbb{Q}) : \mathbb{Z}_5.[Q - (1,-1,0)] \oplus \mathbb{Z}.[Q_{X^{+}}-(1,-1,0)],$ where $Q = (0,1,1)$ and $Q_{X^{+}}=(1,1,0).$

Group Structure of $G \subseteq J_0(77)(\mathbb{Q})$ :
$$ \mathbb{Z}_{10}.(-8D_{\text{tor,1}}+3D_{\text{tor,3}}) \oplus \mathbb{Z}_{60}(31D_{\text{tor,2}}-8D_{\text{tor,3}}) \oplus \mathbb{Z}.D_1 $$
where $$D_{\text{tor,1}} =[(1/2,1/2,1/2,1/2,1/2,1,0)-(1,0,0,0,0,0,0)],$$
$$ D_{\text{tor,2}} = [(5,3,3,2,2,2,1)-(1,0,0,0,0,0,0)],$$ $$D_{\text{tor,3}}=[(5,3,3,2,2,9,1)-(1,0,0,0,0,0,0)],$$
$$D_1=[P_2+ \overline{P_2} -(1/2,1/2,1/2,1/2,1/2,1,0)-(5,3,3,2,2,9,1)]=\rho^{*}([Q_{X^{+}}-(1,-1,0)])$$ 
for $P_2$ defined in the table below satisfying $\rho(Q_{X^{+}}) = P_2.$

Primes used in sieve : 19,23.

There are no quadratic points on $X_0(77)$ that do not come from $X_0(77)^{+}(\mathbb{Q})$ via $\rho: X_0(77) \rightarrow X_0(77)^{+}$ and there is no non-cuspidal rational point. Moreover, we obtain the following table which gives a complete list of all quadratic points up to Galois conjugacy. In addition, the elliptic curves induced by these quadratic points do not have a 77-isogeny in a quadratic extension of the base field.

\begin{table}[h!]
	\centering
	\begin{tabular}{ p{0.75cm} p{0.60cm} p{12.75cm}  p{2cm} p{0.58cm} }
		Name  & \ $w^2$ & \ \ \ \ \ \ \ \ \ \ \ Coordinates & j-invariant & CM  \\
		\hline
		\hline \\
$P_1$ & -19 & $\left[\frac{1}{2}(-w-7), \frac{1}{2}(-3w -3), -w, \frac{1}{2}(-w + 1), -w-1, \frac{1}{2}(-3w-5),1\right]$ & -884736 & -19 \\		
\\				
$P_2$ & -19 & $\left[\frac{1}{2}(w-1), w-1,\frac{1}{2}(w-5),w-2, \frac{1}{2}(3w-1),\frac{1}{2}(3w-1),1\right]$ & -884736 & -19  \\ 
\\
$P_3$ & -7 & $\left[-w-2, \frac{1}{2}(-3w-1), \frac{1}{2}(-w-1), \frac{1}{2}(-w-3), \frac{1}{2}(-3w-3), \frac{1}{2}(-3w-3),1\right]$ & \ \ \ -3375 & -7 \\
\\
$P_4$ & -7 & $\left[w-2, \frac{1}{2}(7w -1), \frac{1}{2}(w-1), \frac{1}{2}(5w - 3), \frac{1}{2}(3w - 3), \frac{1}{2}(7w-3),1\right]$ & 16581375 & -28 \\
\hline		
\end{tabular}
\end{table}

\section{Growth of Torsion} 
The following lemmas are often used in the classification of torsion of quadratic twists. For the first lemma we make use of the Galois equivariance property of the Weil Pairing. 

\begin{lem}\label{lem5}
Let $E/K$ be an elliptic curve, $n$ a positive integer and $K(E[n])$ the field of definition of $E[n]$. Then $K(E[n])$ contains a primitive $n$th root of unity.
\end{lem}

\begin{proof}
Consider a basis $\{e,s\}$ of $E[n]$ which have coordinates in $K(E[n])$. Then $e,s$ are fixed by the action of $\sigma$ for any $\sigma \in$ Gal$(\bar{K}/K(E[n])).$ By the property of the Weil pairing, it follows that $\langle \sigma(e), \sigma(s) \rangle) = \sigma(\langle e, s \rangle) = \langle e, s \rangle.$
As a result $\langle e, s \rangle $  that is a primitive $n$th root of unity is fixed by $\sigma$. This proves the claim.
\end{proof}

\begin{lem}\label{lem6}
Let $\phi$ be the Euler's totient function. If $K$ contains 
a primitive nth root of unity $\mu_n$ and $K/\mathbb{Q}$ is Galois, then $ \phi(n) \leq [K:\mathbb{Q}]$ and $\left(\mathbb{Z}_n\right)^{*} \subseteq \text{Gal}(K/\mathbb{Q}).$
\end{lem}

\begin{proof}
If $\mu_n \in K$ then $\mathbb{Q}(\mu_n) \subseteq K$ and so $ [\mathbb{Q}(\mu_n): \mathbb{Q}] \subseteq [K: \mathbb{Q}].$ Then  $\phi(n)= [\mathbb{Q}(\mu_n) : \mathbb{Q}]$ divides $[K:\mathbb{Q}].$ Now, since $K/\mathbb{Q}$ is Galois, we have Gal$(K/\mathbb{Q}) \supseteq$ Gal$(\mathbb{Q}(\mu_n)/\mathbb{Q}) \simeq \left(\mathbb{Z}_n\right)^*.$
\end{proof}

The next proposition gives a bound when full $p$-torsion may appear over a quadratic extension of certain base fields. We make use of this proposition to narrow the list of possible odd torsion groups.

\begin{prop}\label{prop4}
Let $K$ be an imaginary quadratic field and let $E$ be an elliptic curve over $K.$ If  $L$ is a quadratic extension of $K,$ then the largest odd prime $p$ such that $\mathbb{Z}_p \oplus \mathbb{Z}_p\subseteq E(L) $ is $p=3.$
\end{prop}

\begin{proof}
Suppose $ \mathbb{Z}_p \oplus \mathbb{Z}_p \subseteq E(L)$ As shown in Lemma \ref{lem5}, $L$ has a primitive $p$th root of unity, denoted by $\mu_p$. Then we have $\mathbb{Q}(\mu_p) \subseteq L$ and so 
$ [L:\mathbb{Q}(\mu_p)][\mathbb{Q}(\mu_p) : \mathbb{Q}]= [L : \mathbb{Q}] =4 $. 
This implies that $[\mathbb{Q}(\mu_p) : \mathbb{Q}]$ divides $4.$ Since $[\mathbb{Q}(\mu_p) : \mathbb{Q}]= p-1 $ for any prime $p$ and $p-1$ divides 4, this is only possible when $p$ is 2,3 or 5. If $p=5$,  $L=\mathbb{Q}(\mu_5)$ and Gal$(\mathbb{Q}(\mu_5)/\mathbb{Q})$ is isomorphic to $\mathbb{Z}_4$. However Gal$(L/\mathbb{Q})\simeq \mathbb{Z}_2 \oplus \mathbb{Z}_2$ and $\mathbb{Q}(\mu_5)$ has a unique intermediate field $\mathbb{Q}(\sqrt{5})$ which is not imaginary. For the occurrence of full $3$-torsion, see Remark \ref{rem4}. 
\end{proof}

We now have sufficient tools to examine the growth of odd torsion. From Theorem \ref{mainlist}, some fields in $\mathcal{S}$ have exactly the same list of possible odd torsion subgroups. These fields can be grouped as $\mathcal{S}_1=\{\mathbb{Q}(\sqrt{-11}), \mathbb{Q}(\sqrt{-67}), \mathbb{Q}(\sqrt{-163})\},$  $\mathcal{S}_2=\{\mathbb{Q}(\sqrt{-2}), \mathbb{Q}(\sqrt{-19})\}$ and $\mathcal{S}_3 =\{\mathbb{Q}(\sqrt{-7}), \mathbb{Q}(\sqrt{-43})\}.$

\begin{thm}\label{thmt}
	Let $K$ be in $\mathcal{S}$, $E/K$ an elliptic curve and  let $d$ be a non-square in $K$.
	\begin{enumerate}
		\item If $K \in \mathcal{S}_1 \cup \mathcal{S}_3$ and $E(K)_{\text{tor}} \simeq \mathbb{Z}_{15}$ then $E^d(K)_{\text{tor}} \simeq \mathbb{Z}_1$
		\item If $K \in \mathcal{S}_2 \cup \mathcal{S}_3$ and $E(K)_{\text{tor}} \simeq \mathbb{Z}_{11}$ then  $E^d(K)_{\text{tor}} \simeq \mathbb{Z}_1$
		\item If $E(K) \simeq \mathbb{Z}_9$ when 
		
		(i) $K \in \{\mathbb{Q}(\sqrt{-2}), \mathbb{Q}(\sqrt{-11})\}$ then $E^d(K) \simeq \mathbb{Z}_1$ or $\mathbb{Z}_3$
		\
		
		(ii) Otherwise,  $E^d(K) \simeq \mathbb{Z}_1$
		\item If $E(K)_{\text{tor}} \simeq \mathbb{Z}_7$ then $E^d(K) \simeq \mathbb{Z}_1$
		\item If $E(K)_{\text{tor}} \simeq \mathbb{Z}_5$ then $E^d(K)_{\text{tor}} \simeq \mathbb{Z}_1$ or $\mathbb{Z}_3$
		
		\item If $E(K)_{\text{tor}} \simeq \mathbb{Z}_3$ when
		 
		(i) $K \in \{\mathbb{Q}(\sqrt{-2}), \mathbb{Q}(\sqrt{-11})\}$ then $E^d(K) \simeq \mathbb{Z}_1, \mathbb{Z}_3, \mathbb{Z}_5,$ or $\mathbb{Z}_9$ 
		\
		
		(ii) Otherwise, $E^d(K) \simeq \mathbb{Z}_1, \mathbb{Z}_3,$ or $\mathbb{Z}_5$
		
		\item If $E(K)_{\text{tor}} \simeq \mathbb{Z}_1$ then $E^d(K)_{\text{tor}} \simeq \mathbb{Z}_1, \mathbb{Z}_3, \mathbb{Z}_5, \mathbb{Z}_7, \mathbb{Z}_9, \mathbb{Z}_{11},$ or $\mathbb{Z}_{15}$ (depending on $K)$
	\end{enumerate}
Hence, the groups $\mathbb{Z}_7,$ $\mathbb{Z}_{11}$ and $\mathbb{Z}_{15}$ do not grow in any quadratic extension of $K$.

\end{thm}

\begin{proof}
Note that if $E^d$ is a quadratic twist of $E$, then $E$ is a quadratic twist of $E^d.$ Moreover, by Remark \ref{rem1} if $E(K)_{\text{tor}}$ has odd order then $E^d(K)_{\text{tor}}$  should also have odd order. Then by Proposition \ref{mainlist} both $E(K)_{\text{tor}}$ and $E^d(K)_{\text{tor}}$ belong to the set $\{\mathbb{Z}_1,  \mathbb{Z}_3, \mathbb{Z}_5,  \mathbb{Z}_7, \mathbb{Z}_9,\mathbb{Z}_{11},\mathbb{Z}_{15}\}.$  First, note that if $E(K)_{\text{tor}} \simeq \mathbb{Z}_m$ and $E^d(K) \simeq \mathbb{Z}_m$ for $m \in \{5,7,11\}$, then by Lemma \ref{lem1}, $E(K(\sqrt{d}))$ would contain full $m$-torsion which is impossible by Proposition \ref{prop4}. In particular, if one fixes $E(K)_{\text{tor}} \simeq \mathbb{Z}_5,$ then there are no points of order 7, 9, and 11 in $E^d(K)_{\text{tor}}$ since together with a $K$-rational point of order $5$ on  $E$, this would give rise to a cyclic $35$, $45$ and $55$-isogeny over a quadratic extension of $K$ respectively. But this cannot happen by Theorem \ref{thmc}. Likewise, if  $E(K)_{\text{tor}} \simeq \mathbb{Z}_m$ where $m \in \{ 7,11 \}$ then  since there exists no cyclic $21$ and $33$-isogenies in a quadratic extension of $K,$ it follows that $\mathbb{Z}_3 \not\subseteq E^d(K).$ Furthermore, when $m=7,$ we can conclude that $\mathbb{Z}_{11} \not\subseteq E^d(K)$ due to the non-existence of a cyclic 77-isogeny over a quadratic extension of the base field. 

Having a point of order 9 in $E(K)_{\text{tor}}$ and a point of order $3$ in $E^d(K)$ 
imply that $E$ is isogenous (over $K$) to an elliptic curve $E'$ with a cyclic $27$-isogeny by \cite[Lemma 7]{Naj16}. This is only possible when $K$ is $\mathbb{Q}(\sqrt{-2})$ or $\mathbb{Q}(\sqrt{-11})$ by Theorem \ref{thmc}. On the other hand, if $E(K)_{\text{tor}}\simeq E^d(K)_{\text{tor}} \simeq \mathbb{Z}_9$ then by Lemma \ref{lem1}, $E(K(\sqrt{d}))$ contains full 9-torsion which in turn implies that $K(\sqrt{d})$ has a primitive 9th root $\mu_9$ of unity. This further means that $[K(\sqrt{d}):Q]=4$ is divisible by $[\mathbb{Q}(\mu_9), \mathbb{Q}]|=6$ which gives us a contradiction. If $E(K)_{\text{tor}} \simeq \mathbb{Z}_{15}$ and $E^d(K)_{\text{tor}} \simeq \mathbb{Z}_3$, then by \cite[Lemma7]{Naj16} $E$ is isogenous (over $K$) to an elliptic curve $E'$ which has a cyclic 45-isogeny. Magma computations for $N=45$ in section 3 additionally show that such a curve $E'$ must be  defined over $K=\mathbb{Q}(\sqrt{-11})$ with CM. Since $E'$ is isogenous to $E,$ then $E$ is also CM with $E(K(\sqrt{d}))_{\text{tor}} \simeq \mathbb{Z}_3 \oplus \mathbb{Z}_{15}$ which is impossible by \cite{CCRS14}. Finally, assume $E(K)_{\text{tor}}$ is trivial. Since $K$ does not contain a primitive 3rd root of unity, $E(L)$ doesn't include full 3-torsion, completing the proof.

\end{proof}

\section{Proof of Theorem \ref{mainthm}}

	Let $K$ and $L$ be given as in the hypothesis.  Suppose $E/K$ is an elliptic curve with $E(K)[2]$ trivial. Then $E(L)[2]$ must be trivial as simply observed in Remark \ref{rem1}. Let $p$ be an odd prime and $P \in E(L)_{\text{tor}}$ a point of order $p.$ By Theorem \ref{mainlist}, the possible values of $p$ are 3,5,7,11,13, and 17. Let us assume $p \in \{13,17\}.$ WLOG we may assume  $L=K(\sqrt{d})$ for some non-square $d \in K$ and following Lemma \ref{lem1}, we have either $P \in E(K)[p]$ or $P \in E^d(K)[p].$ Hence, replacing $E$ by $E^d$ if necessary, we may assume $\mathbb{Z}_p \subseteq E(K)[p].$ But this cannot happpen since the groups $\mathbb{Z}_{13}$ and $\mathbb{Z}_{17}$ do not arise as torsion subgroup of $E(K)$ for any $K \in \mathcal{S}$ by Theorem \ref{mainlist}. 
	In the remaining cases, we have that $E(L)_{\text{tor}} \simeq \mathbb{Z}_m \oplus \mathbb{Z}_n$ where $m \in \{3,5,7,9,11,15\}$ and $n \in\{3,5,7,9,11,15\}.$ Note that Proposition \ref{prop4} eliminates the possibilities where $m,n$ are both primes with $m=n > 3.$ Applying Lemma \ref{lem1} to $E(L)_{\text{tor}}$ reduces the problem to showing whether $E(K)_{\text{tor}} \simeq \mathbb{Z}_m$ and $E^d(K)_{\text{tor}} \simeq \mathbb{Z}_n$ can be realized (replacing $E$ by $E^d$ if necessary).
    We answered this question by classifying the possible torsion structures for $E^d(K)_{\text{tor}}$ in Theorem \ref{thmt}, proving the claim.
	
\qed

\bibliographystyle{alpha}
\bibliography{../bibfile}

\end{document}